\documentclass[12pt,draft]{article}
\usepackage{amsmath,amsfonts,amssymb,amsthm,comment}
\usepackage{color}
\usepackage[mathscr]{eucal}

\newcommand{\w}{\omega}
\newcommand{\1}{{\bf 1}}

\DeclareMathOperator{\Com}{Com}

\DeclareMathOperator{\Aut}{Aut\,}

\newcommand{\affine}[2]{{\mathcal L}(#1,#2)}

\newcommand{\inn}{{\rm I}}

\newcommand\Z{\mathbb{Z}}
\newcommand\Zpos{\Z_{\geq0}}
\newcommand\Zplus{\Z_{>0}}

\newcommand\C{\mathbb{C}}

\newcommand\h{\mathfrak{h}}

\newcommand\splin{\mathfrak{sl}}

\newcommand{\NO}{\,{\raise0.25em\hbox{$\mathop{\hphantom {\cdot}}\limits^{_{\circ}}_{^{\circ}}$}}\,}

\newtheorem{theorem}{Theorem}[section]
\newtheorem{proposition}[theorem]{Proposition}
\newtheorem{lemma}[theorem]{Lemma}
\newtheorem{corollary}[theorem]{Corollary}

\theoremstyle{definition}

\theoremstyle{remark}

\numberwithin{equation}{section}

\begin{document}

\begin{center}
\begin{large}
Commutant of $\mathcal{L}_{\widehat{\splin}_2}(4,0)$ in the cyclic permutation orbifold of  
$\mathcal{L}_{\widehat{\splin}_2}(1,0)^{\otimes 4}$
\end{large}
\end{center}

\begin{center}
Toshiyuki Abe\footnote{The first author was partially supported by JSPS Grant-in-Aid for Young Scientists (B) No. 23740022.} 
\\
Faculty of Education, 
Ehime University\\
Matsuyama, Ehime 790-8577, Japan
\\ and 
\\
Hiromichi Yamada\footnote{The second author was partially supported by JSPS Grant-in-Aid for
Scientific Research No. 23540009.}\\
Department of Mathematics, Hitotsubashi University\\
Kunitachi, Tokyo 186-8601, Japan
\end{center}

\begin{center}
{\bf Abstract}
\end{center}
We study the commutant of the vertex operator algebra 
$\mathcal{L}_{\widehat{{\splin}}_2}(4,0)$ in the cyclic permutation orbifold model 
$(\mathcal{L}_{\widehat{{\splin}}_2}(1,0)^{\otimes 4})^\tau$ 
with $\tau=(1\,2\,3\,4)$.  
It is shown that the commutant is isomorphic to a $\Z_2\times\Z_2$-orbifold model 
of a tensor product of two lattice type vertex operator algebras of rank one.  

%\tableofcontents
\section{Introduction}\label{Sect1} 
Let $A_1 = \Z\alpha$, $\langle \alpha, \alpha \rangle = 2$ be a root lattice of 
type $A_1$. 
It is well-known that the vertex operator algebra $V_{A_1}$ 
associated to the lattice $A_1$ is isomorphic to a simple affine vertex operator algebra 
$\mathcal{L}_{\widehat{\splin}_2}(1,0)$ of type $\splin_2$ with level $1$. 
For an integer $k \ge 2$, the cyclic sums of the weight one vectors in 
$V_{A_1}$ in the tensor product $V_{A_1}^{\otimes k}$ 
of $k$ copies of $V_{A_1}$ {generate} a vertex operator subalgebra isomorphic to 
$\mathcal{L}_{\widehat{\splin}_2}(k,0)$. 
The commutant $M$ of $\mathcal{L}_{\widehat{\splin}_2}(k,0)$ in 
$V_{A_1}^{\otimes k}$ 
%$\mathcal{L}_{\widehat{\splin}_2}(1,0)^{\otimes k}$
 has been studied well 
(see for example \cite{JiangLin}, \cite{LamSakuma}, \cite{LamYamada}). 
Among other things the classification of irreducible modules for $M$ 
and the rationality of $M$ were established in \cite{JiangLin}. 

Let $\tau$ be a cyclic permutation on the tensor components of $V_{A_1}^{\otimes k}$ of 
length $k$. Then $\tau$ is an automorphism of the vertex operator algebra $V_{A_1}^{\otimes k}$ 
and 
every element of $\mathcal{L}_{\widehat{\splin}_2}(k,0)$ is fixed by $\tau$. 
Thus $\tau$ induces an automorphism of $M$. 
Our main concern is the orbifold model $M^\tau$ of $M$ by $\tau$, that is, 
the set of fixed points of $\tau$ in $M$, 
which is the commutant of $\mathcal{L}_{\widehat{\splin}_2}(k,0)$ in 
the orbifold model $(V_{A_1}^{\otimes k})^\tau$.

The vertex operator algebra $\mathcal{L}_{\widehat{\splin}_2}(k,0)$ contains a subalgebra 
$T$ isomorphic to a vertex operator algebra associated to a rank one lattice 
generated by a square norm $2k$ element, which corresponds to a Cartan subalgebra 
of $\splin_2(\C)$. 
The commutant of $T$ in $V_{A_1}^{\otimes k}$ is a lattice type vertex operator algebra 
$V_{\sqrt{2}A_{k-1}}$. On the other hand
the commutant $K(\splin_2,k)$ of $T$ in $\mathcal{L}_{\widehat{\splin}_2}(k,0)$ is called 
a parafermion vertex operator algebra of type $\splin_2$. 
The parafermion vertex operator algebra  $K(\splin_2,k)$ has been studied both 
in mathematics and in physics from various points of view 
(see for example \cite{ArakawaLamYamada}, \cite{DongLamYamada09}, 
\cite{DongLamWangYamada10}, \cite{DongLepowsky93}). 
We note that $M \otimes \mathcal{L}_{\widehat{\splin}_2}(k,0) \subset V_{A_1}^{\otimes k}$ and 
$M \otimes K(\splin_2,k) \subset V_{\sqrt{2}A_{k-1}}$. 
In fact, $M$ is the commutant of $K(\splin_2,k)$ in $V_{\sqrt{2}A_{k-1}}$. 

If $k = 2$, then $M$ is isomorphic to the simple Virasoro vertex operator algebra 
$L(\frac{1}{2},0)$ of central charge $\frac{1}{2}$. 
In this case $\tau$ acts trivially on $M$ and $M^\tau$ coincides with $M$. 
The first nontrivial case, that is, the orbifold model $M^\tau$ for the case $k=3$ 
was studied in \cite{DongLamTanabeYamadaYokoyama}. 
It was shown that $M^\tau$ is a $W_3$-algebra of central charge $\frac{6}{5}$. 
Furthermore, the classification of irreducible modules for $M^\tau$ was obtained 
and their properties were discussed in detail. 
Those results were used for the study of the vertex operator 
algebra $(V_{\sqrt{2}A_2})^\tau$ in \cite{TanabeYamada}.

In this paper we consider the orbifold model $M^\tau$ for the case $k=4$. 
The study of $M^\tau$ should lead to a better understanding of the 
structure of $(V_{\sqrt{2}A_3})^\tau$, for $M^\tau \otimes K(\splin_2,4)$ is 
contained in $(V_{\sqrt{2}A_3})^\tau$. 

Let $L = \Z\alpha \oplus \Z\alpha \oplus \Z\alpha$ be an orthogonal sum of three copies 
of $\Z\alpha$, where $\langle \alpha,\alpha \rangle = 2$. 
The main idea is the use of  an automorphism $\rho$ of the vertex operator algebra 
$V_L$ studied in \cite{DongLamYamada01}, 
which maps $V_N$ onto $V_L^+$. 
Here $N$ is a sublattice of $L$ isomorphic to the sublattice $\sqrt{2}A_3$ of $A_1^4$. 
It was shown in \cite{DongLamYamada01} that $\rho(M) = \Com_{V_L^+}(V_{\Z\gamma})$; 
the commutant of $V_{\Z\gamma}$ in $V_L^+$, 
where $\gamma = (\alpha,\alpha,\alpha) \in L$. 

The cyclic permutation $\tau$ on the tensor components of 
$V_{A_1}^{\otimes 4} = V_{A_1^4}$ is a lift of an isometry of the underlying lattice $A_1^4$. 
We denote the isometry of $A_1^4$ by the same symbol $\tau$. 
The sublattice $\sqrt{2}A_3$ of $A_1^4$ is invariant under $\tau$. 
Hence we can discuss an isometry $\widetilde{\tau}$ of $N$ corresponding to 
the isometry $\tau$ of $\sqrt{2}A_3$ by the isomorphism $N \cong \sqrt{2}A_3$. 
We extend $\widetilde{\tau}$ to an isometry of $L$ and consider its lift to an automorphism 
of the vertex operator algebra $V_L$. We denote the automorphism by the same symbol 
$\widetilde{\tau}$. 
Let $\tau' = \rho \widetilde{\tau} \rho^{-1}$ be the conjugate of $\widetilde{\tau}$ 
by $\rho$ so that $\rho(M^\tau) = \Com_{V_L^+}(V_{\Z\gamma})^{\tau'}$. 
It turns out that $\Com_{V_L^+}(V_{\Z\gamma})^{\tau'}$ can be expressed as 
$(V_{\Z\gamma_1} \otimes V_{\Z\gamma_2})^G$, where $\gamma_1$ and $\gamma_2$ 
are elements of $L$ of square norm $12$ and $4$, respectively and $G$ is a group 
of automorphisms of $V_{\Z\gamma_1} \otimes V_{\Z\gamma_2}$ isomorphic to 
$\Z_2\times\Z_2$. 

It is known that the vertex operator algebra $M$ is generated by the 
set of conformal vectors $\omega_\alpha$ of central charge $\frac{1}{2}$ associated 
to the positive roots $\alpha$ of type $A_3$ (see \cite{JiangLin}, \cite{LamSakuma}). 
However, it is difficult to describe the properties of the orbifold model $M^\tau$ 
in terms of those generators of $M$. 
By the result in this paper we can discuss $(V_{\Z\gamma_1} \otimes V_{\Z\gamma_2})^G$ 
instead of $M^\tau$, which seems to be easy to treat.

This paper is organized as follows. 
In Section 2 we review basic materials of vertex operator algebras such as 
conformal vectors and the commutant of a vertex operator subalgebra. 
In Section 3 we discuss two kinds of automorphisms of a vertex operator algebra 
$V_L$ associated to a positive definite even lattice $L$, one is a lift of 
the $-1$-isometry of the lattice $L$ and the other is an exponential of the 
operator $h_{(0)}$ for $h \in \C \otimes_{\Z} L$.  
We also recall three automorphisms of a rank one lattice type vertex operator algebra 
$V_{\Z\alpha}$ studied in \cite{DongLamYamada99}, \cite{DongLamYamada01} 
for $\langle \alpha,\alpha \rangle = 2$. 
In Section 4 we introduce the commutant $M$ of $\mathcal{L}_{\widehat{{\splin}}_2}(4,0)$ 
in $\mathcal{L}_{\widehat{{\splin}}_2}(1,0)^{\otimes 4}$ and its orbifold model $M^\tau$ by $\tau$. 
Finally, in Section 5 we prove that $M^\tau$ is isomorphic to a $\Z_2\times\Z_2$-orbifold model 
of a tensor product of two rank one lattice type vertex operator algebras. 

The automorphism $\rho$ of the vertex operator algebra $V_L$ plays a key role in our argument. 
The use of $\rho$ was suggested by Ching Hung Lam. 
The authors are grateful to him for the important advice.

\section{Preliminaries}\label{Sect2}
In this section we review some basic notions and notations for vertex operator algebras 
(see \cite{MatsuoNagatomo99}, \cite{Kac98}, \cite{LepowskyLi04}).
Let $V = (V, Y, \1, \omega)$ be a vertex operator algebra with the vacuum vector $\1$ 
and the Virasoro vector $\omega$. We denote $\omega$ by $\omega^V$ also.
The $n$-th product of $u,v\in V$ will be written as $u_{(n)}v$ for $n\in\Z$. 
We often regard $u_{(n)}$ as a $\C$-linear endomorphism of $V$. 
Two vectors $u$ and $v$ in $V$ are said to be {\it mutually commutative} 
if $u_{(n)}v=0$ for all $n \in \Zpos$. 
The eigenspace $V_n$ for $L_0=\w^V_{(1)}$ of eigenvalue $n\in\Z$ is finite dimensional. 
A vector in $V_n$ is said to be of weight $n$.

A vertex operator subalgebra of $V$ is a vertex subalgebra $U$ equipped with 
a Virasoro vector $\w^U$. 
When $\w^V=\w^U$, $U$ is said to be full. 
For a pair of a vertex operator algebra $V$ and its subalgebra $U$, the subspace  
\begin{align*}
{\Com}_V(U)=\{v\in V\,|\,u_{(n)}v=0\text{ for }n\in\Zpos\text{ and }u\in U\} 
\end{align*} 
becomes a vertex operator algebra with Virasoro vector $\w^{{\Com}_V(U)}=\w^V-\w^U$. 
We call it the {\it commutant} of $U$ in $V$.  
Actually, it is known that 
\begin{align}\label{conf101}
{\Com}_V(U)=\{v \in V|\w^U_{(0)}v=0\}
\end{align}
(see \cite[Theorem 5.2]{FrenkelZhu92}). 
Hence the commutant of $U$ in $V$ depends only on the Virasoro vector of $U$. 

%Here we will recall the notion of a conformal vector. 
A vector $e\in V_2$ is called a {\it conformal vector} if $L^e_{n}=e_{(n+1)}$, $n\in\Z$ 
give a representation for the Virasoro algebra on $V$ of certain central charge. 
The Virasoro vector of a vertex operator subalgebra $U$ of $V$ is a conformal vector of $V$. 
Let $e$ be a conformal vector in $V$. 
For any vertex operator subalgebra $U$ with $\w^U=e$, 
the commutant $\Com_V(U)$ does not depend on $U$ by \eqref{conf101}. 
In such a case we may write $\Com_V(e)=\Com_V(U)$.  
%As one of consequences of this fact, we have the following proposition.

\begin{proposition}\label{prop102}
Let $V$ be a vertex operator algebra and $e^1, e^2$ mutually commutative conformal vectors in $V$.
Then $\Com_V(e^1+e^2)=\Com_{\Com_V(e^1)}(e^2)$. 
\end{proposition}

\begin{proof}
Let $U$ be a vertex subalgebra generated by $e^1$ and $e^2$. 
Since $e^1$ and $e^2$ are mutually commutative, 
$e_1+e_2$ is the Virasoro vector of $U$. 
Thus we have 
\begin{align*}
\Com_V(e^1+e^2)&=\Com_V(U)\\
&= \{v \in V|e^1_{(n)}v = e^2_{(n)}v = 0\text{ for }n\in\Zpos\}\\
&=\{v \in \Com_{V}(e^1)|e^2_{(n)}v = 0\text{ for }n\in\Zpos\}\\
&=\Com_{\Com_{V}(e^1)}(e^2). 
\end{align*} 
\end{proof} 

An automorphism of a vertex operator algebra $V$ is a linear isomorphism of $V$ 
preserving all $n$-th product, and fixing the vacuum vector $\1$ and the Virasoro vector $\w^V$. 
For a group $G$ consisting of automorphisms of $V$, the subset 
\begin{align*}
V^G=\{v\in V\,|\,g(v)=v\text{ for } g\in G\}
\end{align*}
is a full vertex operator subalgebra of $V$, which is called the orbifold model of $V$ by $G$. 
When $G=\langle \tau\rangle$ is a cyclic group, we denote $V^G$ by $V^\tau$ simply.  

Let $V$ be a vertex operator algebra and $G$ an automorphism group of $V$. 
Let $U$ a vertex operator subalgebra of $V$ and assume that $g(\w^U)=\w^U$ for any $g\in G$. 
Then the restriction $g'$ of $g\in G$ to $\Com_V(U)$ gives rise to an automorphism of $\Com_V(U)$.
In fact, for the automorphism group $H=\{g'|g\in G\}$ of $\Com_V(U)$, we have  
\begin{align}\label{orbcoset}
\Com_V(U)^H=\Com_{V^G}(\w^U). 
\end{align}

\section{Lattice type vertex operator algebras and their automorphisms}\label{Sect3}
In this section we discuss certain automorphisms of lattice type vertex operator algebras. 
Let  $V_L$ be the vertex operator algebra constructed in \cite{FLM} 
for a positive definite even lattice $(L, \langle \,,\, \rangle)$ of rank $d$. 
As a vector space $V_L$ is isomorphic to a tensor product of the symmetric algebra 
$S(\h\otimes t^{-1}\C[t^{-1}] )$ and the twisted group algebra $\C\{L\}$ of $L$, 
where $\h=\C\otimes_{\Z}L$. 
In this paper we only consider the case where $\langle \alpha, \alpha \rangle \in 4\Z$ 
for any $\alpha \in L$ or $L$ is an orthogonal sum of rank one lattices. 
In such a case the central extension $\hat{L}$ of $L$ studied in \cite{FLM} splits 
and $\C\{L\}$ is canonically isomorphic to the ordinary group algebra $\C[L]$. 
Thus we take $\C[L]$ in place of $\C\{L\}$ here. 
A standard basis of $\C[L]$ is denoted by 
$\{ e^\alpha | \alpha \in L\}$ with multiplication $e^\alpha e^\beta = e^{\alpha + \beta}$. 
The vacuum vector of $V_L$ is $\1=1\otimes e^0$, and the Virasoro vector is given by 
\begin{equation*}
\w^{V_L}=\frac{1}{2}\sum_{i=1}^{d}(h_i\otimes t^{-1})^2\otimes e^0,
\end{equation*}
where $\{h_1,\cdots,h_{d}\}$ is an orthonormal basis of $\h$. 
Every eigenvalue for $L_0=\w^{V_L}_{(1)}$ on $V_L$ is a nonnegative integer and 
the eigenspace $(V_L)_n$ with eigenvalue $n$ is finite dimensional. 

\begin{comment}
For a positive definite even lattice $L$ of rank $d$, 
we have a lattice type vertex operator algebra $V_L$ (see \cite{FLM}, \cite{Dong93}). 
As vector spaces $V_L$ is isomorophic to a tensor product of the symmetric algebra 
$S(\h\otimes t^{-1}\C[t^{-1}] )$ and the group algebra $\C[L]=\bigoplus_{\alpha\in L}\C e^\alpha$, 
where $\h=\C\otimes_{\Z}L$ is a vector space equipped with bilinear form 
$\langle-,-\rangle$ over $\C$ which is naturally extended from that of $L$.
The vacuum vector of $V_L$ is $\1=1\otimes e^0$, and the Virasoro vector is given by 
\begin{equation*}
\w^{V_L}=\frac{1}{2}\sum_{i=1}^{d}(h^i\otimes t^{-1})^2\otimes e^0  
\end{equation*}
by means of an orthonormal basis $\{h_1,\cdots,h_{d}\}$ of $\h$. 
Every eigenvalue for $L_0=\w^{V_L}_{(1)}$ is a nonnegative integer and 
the eigenspace $(V_L)_n$ of weight $n$ is finite dimensional. 
\end{comment}

\begin{comment}
We see that the vertex operator algebra $V_L$ is a direct sum of vector subspaces 
$M(1,\alpha)=S(\h\otimes t^{-1}\C[t^{-1}] )\otimes e^{\alpha}$ with $\alpha\in L$, 
where $\h=\C\otimes_{\Z}L$ equipped with bilinear form $\langle-,-\rangle$ over $\C$ 
which is naturally extended from that of $L$.
We see that $M(1)=M(1,0)=S(\h\otimes t^{-1}\C[t^{-1}] )\otimes e^{0}$ is a full 
vertex operator subalgebra of $V_L$ called the free bosonic vertex operator algebra.   
\end{comment}

For simplicity, we regard the sets $L$, $\h$ and $\{e^\alpha |\alpha\in L\}$ 
as subsets of $V_L$, respectively, under the identification
\begin{align*}
\alpha=(\alpha\otimes 1)_{(-1)}\1, \quad h=h_{(-1)}\1,\quad e^\alpha=1\otimes e^\alpha
\end{align*}
for $\alpha\in L$ and $h\in \h$. 
Then we have $(V_L)_0=\C\1$ and 
\begin{align*}
(V_L)_1=\h\oplus\bigoplus_{\stackrel{\alpha\in L}{\langle\alpha,\alpha\rangle=2}}\C e^{\alpha}. 
\end{align*}  
In fact, the weight of $e^\alpha$ is $\frac{1}{2}\langle\alpha,\alpha\rangle$ for $\alpha\in L$.

We will need two kinds of automorphisms of $V_L$.  
One is an involution $\theta_L$ given by a lift of the $-1$-isometry of the lattice $L$. 
We have 
\begin{equation*}
\theta_L(\beta)=-\beta,\quad \theta_L(e^{\beta})=e^{-\beta}
\end{equation*}
for $\beta\in L$. 
The set $(V_L)^{\theta_L}$ of fixed points of $\theta_L$ is also denoted by $V_L^+$. 
The other is an inner automorphism $\inn_{h}=\exp(2\pi\sqrt{-1}h_{(0)})$ for $h\in\h$. 
We have 
\begin{equation*}
\inn_{h}(\beta)=\beta,\quad \inn_{h}(e^{\beta})=e^{2\pi\sqrt{-1}\langle h,\beta\rangle}e^{\beta}  
\end{equation*}
for $\beta\in L$. 
In particular, 
\begin{equation*}
\inn_{\alpha/2\langle\alpha,\alpha\rangle}(e^{m\alpha})=(-1)^me^{m\alpha}
\end{equation*}
for $m\in\Z$ and $0 \ne \alpha\in L$. 
Since $\inn_{h_1} \inn_{h_2} = \inn_{h_1+h_2}$ for $h_1, h_2 \in \h$, the automorphism 
$\inn_{h}$ is of finite order if and only if $h\in \frac{1}{T}L$ for some $T\in\Zplus$. 
By the definition of $\theta_L$ and $\inn_h$, we see that
\begin{equation*}
\theta_L\inn_h\theta_L=\inn_{-h}
\end{equation*}
for any $h\in\h$. 
Therefore, 
\begin{equation}\label{conj102}
\inn_{-\frac{h}{2}}(\inn_h\theta_L)\inn_{\frac{h}{2}}=\theta_L. 
\end{equation} 
That is, $\inn_h\theta_L$ is conjugate to $\theta_L$ in $\Aut(V_L)$.

Let $\Z\alpha$ be a rank one lattice with $\langle\alpha,\alpha\rangle=2$, 
which is the root lattice of type $A_1$. 
In \cite[Section 2]{DongLamYamada99}, 
three involutions $\theta_1$, $\theta_2$ and $\sigma$ 
of $V_{\Z\alpha}$ are considered. 
The involutions $\theta_1$ and $\theta_2$ are expressed as
\begin{equation*}
\theta_1=\inn_{\frac{\alpha}{4}},\quad  
\theta_{2}=\theta_{\Z\alpha},
\end{equation*}
and $\sigma$ is a unique extension of the involution of the Lie algebra 
$(V_{\Z\alpha})_1\cong {\splin}_2(\C)$ given by  
\begin{equation}\label{sigmainvolution}
\sigma(\alpha)=E^\alpha,\quad \sigma(E^\alpha)=\alpha, \quad \sigma(F^\alpha)=-F^\alpha,
\end{equation}
where 
$E^\alpha = e^\alpha + e^{-\alpha}$ and $F^\alpha = e^\alpha - e^{-\alpha}$. 
As automorphisms of $V_{\Z\alpha}$, we have
\begin{equation}\label{conj101}
\sigma\theta_{\Z\alpha}\sigma=\inn_{\frac{\alpha}{4}}. 
\end{equation}

\section{Orbifold model $M^\tau$}\label{Sect4}
In this section we introduce an orbifold model $M^\tau$. 
We use the notation $X_N$ to denote the root lattice of type $X_N$. 
We also write $X_N^i$ 
%\begin{equation*}
%X_N^i=\overbrace{X_N\oplus\cdots \oplus X_N}^{i}
%\end{equation*}
for an orthogonal sum of $i$ copies of the root lattice $X_N$. 

For simplicity, we write $\affine{k}{0}$ for the simple affine vertex operator algebra 
$\mathcal{L}_{\widehat{\splin}_2}(k,0)$ associated to ${\splin}_2(\C)$ of positive integer level $k$. 
It is well-known that $\affine{1}{0}$ is isomorphic to the lattice type vertex operator algebra 
$V_{A_1}$ associated to the root lattice of type $A_1$. 
Thus we have natural isomorphisms 
\begin{equation}\label{isom101}
\affine{1}{0}^{\otimes 4}\cong V_{A_1}^{\otimes 4}\cong V_{A_1^4}
\end{equation}
of vertex operator algebras. 
Let $\tau$ be a cyclic permutation on $\affine{1}{0}^{\otimes 4}$ defined by 
\begin{equation*}
\tau(a_1\otimes a_2\otimes a_3\otimes a_4)=a_4\otimes a_1\otimes a_2\otimes a_3
\end{equation*}
for $a_i\in \affine{1}{0}$. 
%Under the isomorphisms \eqref{isom101}, $\tau$ can be thought to be a lift of an isometry  
In fact, $\tau$ is a lift of an isometry
\begin{equation}\label{eq:tau-A_1^4}
\tau:\alpha_i \mapsto \alpha_{i+1},\quad i\in \Z/4\Z
\end{equation} 
of the lattice 
\begin{equation}\label{eq:A_1^4}
A_1^4=\Z\alpha_1\oplus \Z\alpha_2\oplus \Z\alpha_3\oplus \Z\alpha_4
\end{equation}
with $\langle\alpha_i,\alpha_j\rangle=2\delta_{i,j}$.

Set  
\begin{equation*}
H=\alpha_1+\alpha_2+\alpha_3+\alpha_4, 
\end{equation*}
and consider a sublattice $\sqrt{2} A_3=\sum_{i,j=1}^{4}\Z(\alpha_i-\alpha_j)$ of $A_1^4$.
We note that $\Z H$ and $\sqrt{2} A_3$ are mutually orthogonal and that 
$H\equiv 4\alpha_1$ modulo $\sqrt{2} A_3$. 
Hence $A_1^4 = \cup_{i=0}^3 (i \alpha_1+\Z H+\sqrt{2} A_3)$ and 
\begin{equation*}
V_{A_1^4} = \oplus_{i=0}^3 V_{i \alpha_1+\Z H+\sqrt{2} A_3}
\end{equation*}
as $V_{\Z H+\sqrt{2} A_3}$-modules. It follows that 
\begin{equation*}
{\Com}_{V_{A_1^4}}(V_{\Z H})=V_{\sqrt{2}A_3}. 
\end{equation*}

Since $\tau$ leaves $V_{\Z H}$ invariant, 
$\tau$ induces an automorphism of the vertex operator algebra $V_{\sqrt{2}A_3}$, 
which will be also denoted by $\tau$.
This automorphism is a lift of the restriction of the isometry $\tau$ \eqref{eq:tau-A_1^4} 
of $A_1^4$ to its sublattice $\sqrt{2}A_3$: 
\begin{equation}\label{eq:tau-sqrt2-A3}
\tau(\alpha_i-\alpha_{i+1})=\alpha_{i+1}-\alpha_{i+2},\quad i\in\Z/4\Z.
\end{equation} 
Since $V_{\Z H}$ is contained in $(V_{A_1^4})^{\tau}$, 
it follows from \eqref{orbcoset} that
\begin{equation*}
(V_{\sqrt{2}A_3})^{\tau} = {\Com}_{(V_{A_1^4})^{\tau}}(V_{\Z H})
\end{equation*}

We next take two vectors 
\begin{equation*}
E=e^{\alpha_1}+e^{\alpha_2}+e^{\alpha_3}+e^{\alpha_4},\quad 
F=e^{-\alpha_1}+e^{-\alpha_2}+e^{-\alpha_3}+e^{-\alpha_4}
\end{equation*}
in $\ (V_{A_1^4})_1^\tau$. 
Then the set $\{E,H,F\}$ generates a vertex operator subalgebra $U$ of $(V_{A_1^4})^\tau$ 
isomorphic to $\affine{4}{0}$. 
We consider the commutant 
\begin{equation*}
M={\Com}_{V_{A_1^4}}(U)\cong {\Com}_{\affine{1}{0}^{\otimes 4}}(\affine{4}{0}). 
\end{equation*}

We note that $U$ contains $V_{\Z H}$ as a vertex operator subalgebra.
The commutant 
\begin{equation*}
K_0={\Com}_{U}(V_{\Z H})
\end{equation*}
has been studied in \cite{ArakawaLamYamada}, \cite{DongLamYamada09} 
and \cite{DongLamWangYamada10}. 
The Virasoro vector of $K_0$ is $\omega^{K_0}=\omega^U-\omega^{V_{\Z H}}$.
Since $\omega^{K_0}$ and $\omega^{V_{\Z H}}$ are mutually commutative conformal vectors 
and since $\Com_{V_{A_1^4}}(\w^{V_{\Z H}}) = V_{\sqrt{2} A_3}$, we see that
\begin{equation}\label{eqn1020}
M = {\Com}_{V_{A_1^4}}(\w^U) = {\Com}_{V_{\sqrt{2} A_3}}(\w^{K_0})
\end{equation}
%\begin{equation}
%\begin{split}\label{eqn1020}
%M&={\Com}_{V_{A_1^4}}(\w^U)\\
%&={\Com}_{\Com_{V_{A_1^4}}(\w^{V_{\Z H}})}(\w^{K_0})\\
%&={\Com}_{V_{\sqrt{2} A_3}}(K_0)
%\end{split}
%\end{equation}
by Proposition \ref{prop102}. 
Since $U$ is contained in $(V_{A_1^4})^{\tau}$, we have 
\begin{equation}\label{eqn1021}
M^{\tau}={\Com}_{(V_{\sqrt{2} A_3})^{\tau}}(\w^{K_0})=\left(\Com_{V_{\sqrt{2} A_3}}(\w^{K_0})\right)^{\tau}. 
\end{equation}

%\section{An isomorphism between $M^\tau$ and a certain orbifold model}
\section{Main results}
In this section we show that $M^{\tau}$ is isomorphic to a $\Z_2\times\Z_2$-orbifold model 
of a tensor product of two lattice type vertex operator algebras of rank one. 
Following \cite{DongLamYamada01}, we study an isomorphism between 
the vertex operator algebras $V_{\sqrt{2}A_3}$ and $V_{A_1^3}^+$. 
We remark that another isomorphism was considered in \cite{DongLamYamada99} 
(see \cite[Remark 3.2]{DongLamYamada01}). 
%It is known that this isomorphism can be generalized to an isomorphisms 
%between $V_{\sqrt{2}Q(D_n)}$ and $V_{Q^n(A_1)}^+$ for $n$ (see \cite{DongLamYamada01}). 

Throughout this section, let 
\begin{equation*}
L = \Z\alpha \oplus \Z\alpha \oplus \Z\alpha
\end{equation*}
be an orthogonal sum of 
three copies of $\Z\alpha$, where $\langle \alpha, \alpha \rangle = 2$. 
Set
\begin{equation*}
\alpha^{(1)} = (\alpha,0,0), \quad \alpha^{(2)} = (0,\alpha,0), \quad \alpha^{(3)} = (0,0,\alpha),
\end{equation*}
so that $L = \Z\alpha^{(1)} + \Z\alpha^{(2)} + \Z\alpha^{(3)}$ and 
$\langle \alpha^{(i)}, \alpha^{(j)} \rangle = 2\delta_{i,j}$. 
The vertex operator algebra $V_L$ is isomorphic to 
the tensor product $V_{A_1}^{\otimes 3}$ of three copies of $V_{A_1} = V_{\Z\alpha}$. 
Hence the involution $\sigma$ of $V_{\Z\alpha}$ defined in \eqref{sigmainvolution} 
induces naturally an involution $\sigma\otimes \sigma\otimes\sigma$ of $V_L$. 
Let 
\begin{equation}\label{eq:rho}
\rho=\inn_{\frac{1}{4}(\alpha^{(2)}+\alpha^{(3)})}(\sigma\otimes \sigma\otimes\sigma) 
\in {\rm Aut}(V_L)
\end{equation}
be a composite of $\sigma\otimes \sigma\otimes\sigma$ and the inner automorphism 
$\inn_{\frac{1}{4}(\alpha^{(2)}+\alpha^{(3)})}$ of $V_L$ with respect to 
$\frac{1}{4}(\alpha^{(2)}+\alpha^{(3)}) \in \frac{1}{4}L$ 
(see \cite[Section 3]{DongLamYamada01}).  
We note that $\inn_{\frac{1}{4}(\alpha^{(2)}+\alpha^{(3)})} 
= 1 \otimes \inn_{\frac{1}{4}\alpha} \otimes \inn_{\frac{1}{4}\alpha}$. 
By the definition of $\rho$, we have
\begin{alignat}{4}
\rho(\alpha^{(1)}) &= E^{\alpha^{(1)}}, & \quad \rho(\alpha^{(i)}) &= -E^{\alpha^{(i)}} 
&\quad & (i=2,3),\label{eqn001}\\
\rho(E^{\alpha^{(i)}}) &= \alpha_i & & &\quad&(i=1,2,3),\\
\rho(F^{\alpha^{(1)}}) &= -F^{\alpha^{(1)}}, &\quad \rho(F^{\alpha^{(i)}}) &= F^{\alpha_i} 
&\quad & (i=2,3)\label{eqn002}. 
\end{alignat}

Set 
\begin{gather*}
\beta_1=\alpha^{(1)}+\alpha^{(2)},\quad \beta_2=-\alpha^{(2)}+\alpha^{(3)},
\quad \beta_3=-\alpha^{(1)}+\alpha^{(2)},\\
\gamma=\alpha^{(1)}+\alpha^{(2)} +\alpha^{(3)}.
\end{gather*} 

Then $\{\frac{1}{\sqrt{2}}\beta_1,\frac{1}{\sqrt{2}}\beta_2,\frac{1}{\sqrt{2}}\beta_3\}$ 
forms the set of simple roots of type $A_3$. 
We consider the sublattice  
\begin{equation*}
N=\sum_{i,j=1}^3\Z(\alpha^{(i)}\pm\alpha^{(j)}) =\langle \beta_1,\beta_2,\beta_3\rangle_{\Z}
\end{equation*} 
of $L$. 
It is known that $N \cong \sqrt{2}D_3$ is isomorphic to the sublattice $\sqrt{2}A_3$ 
of the lattice $A_1^4$ \eqref{eq:A_1^4} discussed in Section 4 
by the correspondence 
\begin{equation}\label{equiv101}
\beta_1\leftrightarrow \alpha_1-\alpha_2,\quad \beta_2\leftrightarrow\alpha_2-\alpha_3,
\quad \beta_3\leftrightarrow \alpha_3-\alpha_4.
\end{equation}

This isomorphism between the lattices $N$ and $\sqrt{2}A_3$ induces an isomorphism 
between the vertex operator algebras $V_N$ and $V_{\sqrt{2}A_3}$. 
Thus we can think of the vertex operator subalgebras $M$ and $K_0$ of 
$V_{\sqrt{2}A_3}$ discussed in Section 4 as vertex operator subalgebras of $V_N$. 

We need the following facts in \cite{DongLamYamada01}. 

\begin{theorem}\label{theorem001}
$(1)$ $\rho(V_N)=V_L^+$.

$(2)$ $\rho(\w^{K_0})=\w^{V_{\Z \gamma}}$.

$(3)$ $\rho(M)=\Com_{V_L^+}(\w^{V_{\Z \gamma}})$.  
\end{theorem}

\begin{proof}
The assertions (1) and (2) follow from \cite[Lemmas 3.4 (3) and 3.3 (3)]{DongLamYamada01}.
Then the assertion (3) follows from \eqref{eqn1020}. 
\end{proof}

Under the isomorphism \eqref{equiv101} between $N$ and ${\sqrt{2}A_3}$, 
the isometry $\tau$ \eqref{eq:tau-sqrt2-A3} of the lattice  ${\sqrt{2}A_3}$ corresponds to 
an isometry 
\begin{equation*}
\beta_1\mapsto \beta_2\mapsto \beta_3\mapsto -\alpha^{(2)}-\alpha^{(3)}\mapsto \beta_1
\end{equation*} 
of the lattice $N$. 
This isometry of $N$ is the restriction of an isometry $\widetilde{\tau}$ of the lattice $L$ 
of order $4$ given by 
\begin{equation*}
\widetilde{\tau} : \alpha^{(1)} \mapsto \alpha^{(3)} \mapsto -\alpha^{(1)} \mapsto 
-\alpha^{(3)} \mapsto \alpha^{(1)},\quad \alpha^{(2)} \leftrightarrow -\alpha^{(2)}.
\end{equation*}
The isometry $\widetilde{\tau}$ of $L$ lifts to an automorphism of 
the vertex operator algebra $V_L$ of order $4$, 
which is also denoted by $\widetilde{\tau}$. 

Actually, 
\begin{equation*}
\widetilde{\tau}=(\theta_{\Z\alpha}\otimes \theta_{\Z\alpha}\otimes 1)t_{13}
\end{equation*}
is a composite of $t_{13}$ and $\theta_{\Z\alpha}\otimes \theta_{\Z\alpha}\otimes 1$, 
where $t_{13}$ denotes the transposition of the first component and the third one 
of the tensor product $V_L=V_{A_1}^{\otimes 3}$. 
By direct calculations, we have 
\begin{alignat}{4}
\widetilde{\tau}(\alpha^{(1)}) &= \alpha^{(3)}, 
& \widetilde{\tau}(\alpha^{(2)}) &= -\alpha^{(2)}, 
& \widetilde{\tau}(\alpha^{(3)}) &= -\alpha^{(1)},\label{eqn101}\\
\widetilde{\tau}(e^{\pm \alpha^{(1)}}) &= e^{\pm \alpha^{(3)}}, 
&\quad \widetilde{\tau}(e^{\pm \alpha^{(2)}}) &= e^{\mp \alpha^{(2)}}, 
&\quad \widetilde{\tau}(e^{\pm \alpha^{(3)}}) &= e^{\mp \alpha^{(1)}}.\label{eqn102}
\end{alignat}

Now we consider the conjugate $\tau'$ of $ \widetilde{\tau}$ by $\rho$ \eqref{eq:rho}.
\begin{equation*}
\tau'=\rho \widetilde{\tau} \rho^{-1}\in \Aut(V_L). 
\end{equation*}
%Then we have the following lemma. 

\begin{lemma}\label{lemma002} 
As automorphisms of $V_L$, we have
\begin{equation}
\tau'=(1\otimes\inn_{\frac{1}{4}\alpha}\otimes\inn_{\frac{1}{4}\alpha}) t_{13}. 
\end{equation}
\end{lemma}

\begin{proof}
Using \eqref{conj101}, we can calculate as follows.  
\begin{align*}
\tau'&=\rho \widetilde{\tau} \rho^{-1}\\
&=\inn_{\frac{1}{4}(\alpha_2+\alpha_3)}\sigma\widetilde{\tau}\sigma\inn_{\frac{1}{4}(\alpha_2+\alpha_3)}\\
&=(1\otimes\inn_{\frac{1}{4}\alpha}\otimes \inn_{\frac{1}{4}\alpha})
(\sigma\otimes \sigma\otimes\sigma)(\theta_{\Z\alpha}\otimes \theta_{\Z\alpha}\otimes 1)t_{13}\\
&\quad\circ(\sigma\otimes \sigma\otimes\sigma)
(1\otimes\inn_{\frac{1}{4}\alpha}\otimes \inn_{\frac{1}{4}\alpha})\\
&=(\sigma\theta_{\Z \alpha}\otimes\inn_{\frac{1}{4}\alpha}\sigma\theta_{\Z \alpha}
\otimes \inn_{\frac{1}{4}\alpha}\sigma) t_{13} 
(\sigma\otimes\sigma \inn_{\frac{1}{4}\alpha}\otimes \sigma\inn_{\frac{1}{4}\alpha})\\
&=(\sigma\theta_{\Z \alpha}\sigma\inn_{\frac{1}{4}\alpha}
\otimes\inn_{\frac{1}{4}\alpha}\sigma\theta_{\Z \alpha}\sigma\inn_{\frac{1}{4}\alpha}
\otimes \inn_{\frac{1}{4}\alpha}) t_{13}\\
&=(1\otimes\inn_{\frac{1}{4}\alpha}\otimes \inn_{\frac{1}{4}\alpha}) t_{13}.
\end{align*}
\end{proof}

Recall that $\gamma=\alpha^{(1)}+\alpha^{(2)}+\alpha^{(3)}$. We set
\begin{equation*}
\gamma_1=\beta_2-\beta_3=\alpha^{(1)}-2\alpha^{(2)}+\alpha^{(3)},
\quad \gamma_2=-\beta_2-\beta_3=\alpha^{(1)}-\alpha^{(3)},
\end{equation*}
and consider $P=\Z\gamma_1+\Z\gamma_2+\Z\gamma$. 
Since $\gamma_1$, $\gamma_2 $ and $\gamma $ are mutually orthogonal, we have 
$V_P=V_{\Z\gamma_1}\otimes V_{\Z\gamma_2}\otimes V_{\Z\gamma}$. 
One can easily see that 
\begin{equation}\label{P001}
P=\langle \alpha^{(1)}-\alpha^{(3)},\alpha^{(2)}+2\alpha^{(3)}3,6\alpha^{(3)}\rangle_{\Z}.
\end{equation}
Hence we have a coset decomposition 
\begin{equation*}
L = \cup_{i=0}^5 (i\alpha^{(3)}+P).
\end{equation*}
Since $\alpha^{(3)}=\frac{1}{6}\gamma_1-\frac{1}{2}\gamma_2+\frac{1}{3}\gamma$, 
we have 
%$3\alpha^{(3)}+P=\frac{1}{2}\gamma_1+\frac{1}{2}\gamma_2+P$.
\begin{equation*}
3\alpha^{(3)}+P = \Big(\frac{1}{2}\gamma_1 + \Z\gamma_1\Big) 
+ \Big(\frac{1}{2}\gamma_2 + \Z\gamma_2\Big) + \Z\gamma.
\end{equation*}
Therefore,   
\begin{equation}\label{eq:Comm-VA_1^3-V_gamma}
\begin{split}
{\rm Com}_{V_L}(V_{\Z\gamma})
&= {\rm Com}_{V_{P}\oplus V_{3\alpha^{(3)}+P}}(V_{\Z\gamma})\\
&= \big( V_{\Z\gamma_1}\otimes V_{\Z\gamma_2} \big)
\oplus \big( V_{\frac{1}{2}\gamma_1+\Z\gamma_1}\otimes V_{\frac{1}{2}\gamma_2+\Z\gamma_2}\big). 
\end{split}
\end{equation}

\begin{lemma}\label{lemma003}
$(1)$ The eigenvalues for $\tau'$ on $V_{\Z\gamma_1}\otimes V_{\Z\gamma_2}$ are $\pm 1$.

$(2)$ The eigenvalues for $\tau'$ on 
$V_{\frac{1}{2}\gamma_1+\Z\gamma_1}\otimes V_{\frac{1}{2}\gamma_2+\Z\gamma_2}$ 
are $\pm\sqrt{-1}$.

$(3)$ $(V_{P}\oplus V_{3\alpha^{(3)}+P})^{\tau'}
=(V_{\Z\gamma_1}\otimes V_{\Z\gamma_2})^{\tau'}\otimes V_{\Z\gamma}$. 

\end{lemma}
\begin{proof}
Recall that we regard $L$ as a subset of $V_L$. 
Under the canonical identification between $V_L$ and $V_{A_1}^{\otimes 3}$, we have 
\begin{align*}
\gamma_1&=\alpha\otimes \1\otimes\1-2(\1\otimes \alpha\otimes\1)+\1\otimes\1\otimes \alpha,\\
\gamma_2&=\alpha \otimes \1\otimes\1-\1\otimes\1\otimes \alpha ,\\
\gamma&=\alpha \otimes \1\otimes\1+\1\otimes \alpha \otimes\1+\1\otimes\1\otimes \alpha ,\\
e^{\pm\gamma_1}&=e^{\pm\alpha }\otimes e^{\mp 2\alpha }\otimes e^{\pm\alpha },\\
e^{\pm\gamma_2}&=e^{\pm\alpha }\otimes \1 \otimes e^{\mp\alpha },\\ 
e^{\pm\gamma}&=e^{\pm\alpha }\otimes e^{\pm\alpha}\otimes e^{\pm\alpha  },
\end{align*}
respectively. 
Then by Lemma \ref{lemma002}, we have 
\begin{alignat}{4}
\tau'(\gamma_1)&=\gamma_1,&\quad \tau'(e^{\pm\gamma_1})&=-e^{\pm\gamma_1}, \label{eq201}\\
\tau'(\gamma_2)&=-\gamma_2,&\quad \tau'(e^{\pm\gamma_2})&=-e^{\mp\gamma_2}, \label{eq2012}\\
\tau'(\gamma)&=\gamma,&\quad \tau'(e^{\pm\gamma})&=e^{\pm\gamma}.\label{eq202} 
\end{alignat}
Since $V_{\Z\gamma_1}\otimes V_{\Z\gamma_2}$ is generated by the set 
$\{\gamma_1,\gamma_2,e^{\pm\gamma_1},E^{\gamma_2},F^{\gamma_2}\}$ consisting of 
eigenvectors for $\tau'$ whose eigenvalues are $\pm1$, we have the assertion (1). 

We note that $V_{\frac{1}{2}\gamma_1+\Z\gamma_1}\otimes V_{\frac{1}{2}\gamma_2+\Z\gamma_2}$ 
is an irreducible $V_{\Z\gamma_1}\otimes V_{\Z\gamma_2}$-module and so it is generated by a nonzero vector 
$u=e^{\frac{1}{2}(\gamma_1+\gamma_2)}+\sqrt{-1}e^{\frac{1}{2}(\gamma_1-\gamma_2)}$. 
Since 
\begin{align*}
\tau'(e^{\frac{1}{2}(\gamma_1+\gamma_2)})
&=\tau'(e^{\alpha^{(1)}-\alpha^{(2)}})=e^{-\alpha^{(2)}+\alpha^{(3)}}
=e^{\frac{1}{2}(\gamma_1-\gamma_2)},\\
\tau'(e^{\frac{1}{2}(\gamma_1-\gamma_2)})
&=\tau'(e^{-\alpha^{(2)}+\alpha^{(3)}})=-e^{\alpha^{(1)}-\alpha^{(2)}}
=-e^{\frac{1}{2}(\gamma_1+\gamma_2)}
\end{align*}
we have $\tau'(u)=\sqrt{-1}u$.
Hence by (1), the assertion (2) holds. 

The assertion (3) follows from (1), (2) and \eqref{eq202}.  
\end{proof}

Here we note that $\langle\gamma_1,\gamma_1\rangle=12$ and 
$\langle\gamma_2,\gamma_2\rangle=4$. 
Let 
\begin{equation*}
g = \inn_{\frac{1}{24}\gamma_1}\otimes(\inn_{\frac{1}{8}\gamma_2}\theta_{\Z\gamma_2}) 
\in \Aut(V_{\Z\gamma_1} \otimes V_{\Z\gamma_2}).
\end{equation*}
Then by \eqref{eq201} and \eqref{eq2012}, the restriction of $\tau'$ to 
$V_{\Z\gamma_1}\otimes V_{\Z\gamma_2}$ coincides with the automorphism $g$. 
Hence \eqref{eq:Comm-VA_1^3-V_gamma} and Lemma \ref{lemma003} imply 
the following proposition. 

\begin{proposition}\label{prop104}
Let $L$ and $\tau'$ be as above. 
Then
\begin{equation*}
\Com_{V_L}(V_{\Z\gamma})^{\tau'} = (V_{\Z\gamma_1}\otimes V_{\Z\gamma_2})^g.
\end{equation*}
\end{proposition}

Let $M$ and $\tau$ be as in Section 4. Then we see that
\begin{align*}
\rho(M^\tau) 
&= \Com_{V_L^+}(V_{\Z \gamma})^{\tau'}\\
&= \Com_{V_L}(V_{\Z \gamma})^{\langle \tau', \theta' \rangle}\\
&= (V_{\Z\gamma_1}\otimes V_{\Z\gamma_2})^{\langle g, \theta' \rangle}
\end{align*}
by the definition of $\tau'$, Theorem \ref{theorem001} and Proposition \ref{prop104}, 
where
\begin{equation*}
\theta' = \theta_{\Z\gamma_1} \otimes \theta_{\Z\gamma_2} 
\in \Aut(V_{\Z\gamma_1} \otimes V_{\Z\gamma_2}).
\end{equation*}

Let $G = {\langle g, \theta' \rangle}$ be a subgroup of 
$\Aut(V_{\Z\gamma_1} \otimes V_{\Z\gamma_2})$ generated by $g$ and $\theta'$. 
We note that  
\begin{equation*}
G = \langle (\inn_{\frac{1}{24}\gamma_1}\theta_{\Z\gamma_1})\otimes \inn_{\frac{1}{8}\gamma_2}, 
\inn_{\frac{1}{24}\gamma_1} \otimes(\inn_{\frac{1}{8}\gamma_2}\theta_{\Z\gamma_2})\rangle 
\cong \Z_2 \times \Z_2.
\end{equation*}

\begin{theorem}\label{main101}
Let $M=\Com_{V_{A_1^4}}(\affine{4}{0})$ and $\tau$ a cyclic permutation of $V_{A_1^4}$ of length $4$ 
as in Section 4. Let $G$ be as above. 
Then $M^\tau$ is isomorphic to the orbifold model $(V_{\Z\gamma_1}\otimes V_{\Z\gamma_2})^{G}$.
\end{theorem}

Let
\begin{equation*}
f=\inn_{\frac{1}{48}\gamma_1}\otimes\inn_{\frac{1}{16}\gamma_2}
\in\Aut(V_{\Z\gamma_1}\otimes V_{\Z\gamma_2}).
\end{equation*}
By \eqref{conj102}, we have 
\begin{align*}
f^{-1}((\inn_{\frac{1}{24}\gamma_1}\theta_{\Z\gamma_1})\otimes\inn_{\frac{1}{8}\gamma_2})f
&=\theta_{\Z\gamma_1}\otimes\inn_{\frac{1}{8}\gamma_2},\\
f^{-1}(\inn_{\frac{1}{24}\gamma_1}\otimes(\inn_{\frac{1}{8}\gamma_2}\theta_{\Z\gamma_2}))f
&=\inn_{\frac{1}{24}\gamma_1}\otimes\theta_{\Z\gamma_2}. 
\end{align*}

Let $g_1 = \theta_{\Z\gamma_1}\otimes\inn_{\frac{1}{8}\gamma_2}$ and 
$g_2 = \inn_{\frac{1}{24}\gamma_1}\otimes \theta_{\Z\gamma_2}$. 
Then $\langle g_1,g_2 \rangle = f^{-1} G f \cong \Z_2 \times \Z_2$ and the following corollary holds.

\begin{corollary}\label{main102}
The vertex operator algebra $M^\tau$ is isomorphic to the orbifold model 
$(V_{\Z\gamma_1}\otimes V_{\Z\gamma_2})^{\langle g_1,g_2 \rangle}$.
\end{corollary}

\end{document}